\documentclass[preprint,12pt]{elsarticle}

\usepackage[utf8]{inputenc}
\usepackage[T1]{fontenc}
\usepackage{amsmath,amssymb,amsthm,mathtools}
\usepackage{enumitem}
\usepackage{array}
\usepackage{hyperref}

\hypersetup{hidelinks}

\theoremstyle{plain}
\newtheorem{theorem}{Theorem}[section]
\newtheorem{lemma}[theorem]{Lemma}
\newtheorem{proposition}[theorem]{Proposition}

\theoremstyle{definition}
\newtheorem{definition}[theorem]{Definition}

\newtheorem{counterexample}[theorem]{Counterexample}

\theoremstyle{remark}
\newtheorem{remark}[theorem]{Remark}

\newcommand{\R}{\mathbb{R}}

\newcommand{\CM}{\mathrm{CM}}
\newcommand{\BF}{\mathrm{BF}}

\makeatletter
\def\ps@pprintTitle{%
  \let\@oddhead\@empty
  \let\@evenhead\@empty
  \let\@oddfoot\@empty
  \let\@evenfoot\@empty}
\makeatother

\begin{document}

\begin{frontmatter}

\title{Riccati Reductions for Modified Bessel Ratios:\\
Bernstein Positivity, Exact Certificates, and Transfer Obstructions}
\author[addr1]{Domingos S. P. Salazar\corref{cor1}}
\address[addr1]{Unidade de Educa\c{c}\~ao a Dist\^ancia e Tecnologia,
Universidade Federal Rural de Pernambuco,
52171-900 Recife, Pernambuco, Brazil}
\cortext[cor1]{Corresponding author.}

\begin{abstract}
Several open inequalities for ratios and logarithmic derivatives of
the modified Bessel functions \(I_\nu\) of the first kind and \(K_\nu\) of
the second kind reduce to sign questions for quadratic Riccati expressions.
We isolate this reduction and use it in two directions. First, for the quotient
\(W_\nu(z)=zI_\nu(z)/I_{\nu+1}(z)\), the canonical product for \(I_{\nu+1}\)
yields the partial fraction
\(W_\nu(\sqrt s)=2(\nu+1)+2\sum_{n\ge1}s/(s+j_{\nu+1,n}^2)\), where
\(j_{\nu+1,n}\) is the \(n\)-th positive zero of \(J_{\nu+1}\).
Consequently \(x\mapsto W_\nu(x^\tau)\) is a Bernstein function for
\(\nu>-1\) and \(0<\tau\le1/2\), and this positive exponent range is sharp.
Second, an exact rational certificate at \((\nu,u)=(0,10)\) places
\(I_1(10)/I_0(10)\) below \(0.949\). This refutes the log-concavity question
of Baricz, Ponnusamy, and Vuorinen for \(u\mapsto \sqrt u\,I_\nu(u)\) and its
displayed Riccati reformulations.
The same framework completes the monotonicity classification of
\(K_\nu'/K_\nu^2\), refutes Baricz--Ponnusamy--Vuorinen Question~7 at
\(\nu=1/2\), and gives an entire counterexample to Baricz's
coefficient-ratio complete-monotonicity transfer problem.
\end{abstract}

\begin{keyword}
modified Bessel functions \sep Bernstein functions \sep completely monotone functions \sep logarithmic concavity \sep Tur\'an-type inequalities \sep Riccati equations
\MSC[2020] 33C10 \sep 26A48 \sep 26A51 \sep 26D07 \sep 44A10 \sep 60G51
\end{keyword}

\end{frontmatter}

\section{Introduction}

Ratios and logarithmic derivatives of \(I_\nu\) and \(K_\nu\), the standard
modified Bessel functions of the first and second kinds of order \(\nu\),
appear throughout analysis and probability. They
control condition numbers in the backward recurrences of Amos \cite{amos},
sharp ratio bounds and Tur\'an-type inequalities
\cite{turan,nasell,simpson-spector,ifantis-siafarikas,laforgia-natalini,
baricz-turan-modified,joshi-bissu,segura,hornik-gruen-2013,
ruiz-antolin-segura,hornik-gruen-2024}, normal distribution functions on
spheres and related probabilistic models \cite{hartman-watson,robert}, and
complete-monotonicity and infinite-divisibility results of Ismail and Grosswald
\cite{ismail,ismail-1990,ismail-kelker,grosswald}. Bernstein functions provide the
probabilistic bridge: they are Laplace exponents of possibly killed subordinators
\cite{schilling-song-vondracek}, so a Bernstein certificate for a Bessel ratio
is a reusable positivity statement for subordinated kernels and for
Bessel-ratio inequalities that enter such kernels. A well-known
survey of open functional inequalities for these functions is the question
list of Baricz, Ponnusamy, and Vuorinen \cite{bpv}, hereafter BPV; a companion
source of open problems is Baricz's study of bounds via coefficient ratios
\cite{baricz-bounds}. Yang and Tian \cite[Conjecture~3]{yang-tian} formulate a
conjecture for the Bernstein-function behavior of \(W_\nu(x^{1/\theta})\)
in the range \(\theta\ge2\); below we state and prove the corresponding
Bernstein-property formulation with \(\tau=1/\theta\).
This sits in a recent line of monotonicity and convexity results for modified
Bessel ratios and products
\cite{yang-zheng-k,yang-zheng-i,segura-2021,yang-chu-k,mao-tian-k,yang-tian}.

The present paper settles a group of these problems with a single organizing
device. The modified Bessel equation converts second-order information about
\(I_\nu\) or \(K_\nu\) into a \emph{quadratic} polynomial in the logarithmic
derivative; every monotonicity or concavity statement treated here is the
assertion that this quadratic has a fixed sign on its domain.
For the overview, write \(W_\nu(z)=zI_\nu(z)/I_{\nu+1}(z)\),
\(r_\nu(u)=I_\nu'(u)/I_\nu(u)\), \(q_\nu(u)=I_{\nu+1}(u)/I_\nu(u)\), and
\(y_\nu(u)=uK_\nu'(u)/K_\nu(u)\); these ratios are defined formally in
Section~\ref{sec:prelim}.

\begin{quote}
\emph{After one substitution through the modified Bessel equation, each claim
below is a sign statement for a quadratic polynomial in a logarithmic
derivative. When the claim is true, the quadratic is controlled globally by a
pole--zero representation; when it is false, it is refuted locally by an exact
rational enclosure at a single point.}
\end{quote}

The paper is organized by mechanism: zero representation, rational certificate,
endpoint classification, and transfer obstruction.

\begin{table}[t]
\centering
\caption{Main mechanisms and source problems.}
\label{tab:mechanisms}
\footnotesize
\begin{tabular}{@{}>{\raggedright\arraybackslash}p{0.40\linewidth}
>{\raggedright\arraybackslash}p{0.27\linewidth}
>{\raggedright\arraybackslash}p{0.24\linewidth}@{}}
\hline
Source claim & Mechanism & Output \\
\hline
\(W_\nu(x^\tau)\) Bernstein, Yang--Tian range \cite{yang-tian} &
Zero partial fraction & Bernstein iff \(0<\tau\le\frac12\) \\
\(\sqrt u\,I_\nu(u)\) log-concavity, BPV Question~1 \cite{bpv} &
Riccati certificate at \((0,10)\) & refuted \\
Riccati reformulation of BPV Question~1 &
same certificate & refuted \\
Quadratic \(I_{\nu+1}/I_\nu\) reformulation &
Riccati algebra & refuted \\
\(K_\nu'/K_\nu^2\), BPV Theorem~2(a) and Question~2 \cite{bpv} &
endpoint signs & classified at \(|\nu|=1\) \\
\(u^2K_\nu'(u)\) on \((0,2)\), BPV Question~7 \cite{bpv} &
half-order closed form & refuted \\
Coefficient-ratio complete-monotonicity transfer, Baricz Problem~4.3(c) \cite{baricz-bounds} &
second-coefficient obstruction & refuted \\
\hline
\end{tabular}
\end{table}

\emph{Bernstein positivity from the zero expansion}
(Section~\ref{sec:bernstein-positivity}). Writing
\(W_\nu(z)=zI_\nu(z)/I_{\nu+1}(z)\), the canonical product of \(I_{\nu+1}\)
over the zeros of \(J_{\nu+1}\) gives the partial fraction
\(W_\nu(\sqrt s)=2(\nu+1)+2\sum_{n\ge1}s/(s+j_{\nu+1,n}^2)\), a Bernstein
function of \(s\). Composition closure then proves the Bernstein-property
formulation corresponding to the Yang--Tian range: \(x\mapsto W_\nu(x^\tau)\)
is Bernstein for every
\(\nu>-1\) and \(0<\tau\le1/2\)
(Theorem~\ref{thm:yang-tian}). A two-line endpoint argument shows the
exponent range is sharp for this Bernstein property
(Proposition~\ref{prop:sharpness}).

\emph{One certificate for BPV Question~1}
(Section~\ref{sec:certificate-cluster}). The Riccati normal form
(Lemma~\ref{lem:riccati-I}) identifies strict log-concavity of
\(\sqrt u\,I_\nu(u)\), the corresponding Riccati inequality, and the
corresponding quadratic lower bound for \(I_{\nu+1}/I_\nu\) as the same sign
statement. A single exact rational enclosure places \(I_1(10)/I_0(10)\)
below \(0.949\) (Lemma~\ref{lem:certificate}) and makes that statement fail
at \((\nu,u)=(0,10)\) with a certified positive margin, refuting
BPV Question~1 and its algebraic reformulations
(Counterexamples~\ref{cx:sqrt-log-concavity}--\ref{cx:quadratic-bound}).

\emph{The Bessel-\(K\) quotient: classification and the endpoint question}
(Section~\ref{sec:besselk}). For \(g_\nu=K_\nu'/K_\nu^2\), BPV proved strict
decrease for \(|\nu|\ge1\) and asked about the remaining range. An endpoint
sign analysis of the same Riccati-type quadratic shows \(g_\nu\) increases
near \(0\) and decreases near \(\infty\) for every \(|\nu|<1\), completing the
monotonicity classification (Theorem~\ref{thm:K-classification}). At the
half-order line the unique critical point is explicit, \((\sqrt2-1)/2\)
(Proposition~\ref{prop:halforder-unimodal}). The half-order closed form also
settles BPV Question~7 negatively: \(u^2K_{1/2}'(u)\) is increasing on
\(((1+\sqrt2)/2,\,2)\) (Counterexample~\ref{cx:question7}).

\emph{The limits of transfer} (Section~\ref{sec:transfer}). Baricz's
coefficient-ratio problem asks whether complete monotonicity of
\(\{a_n/b_n\}\) forces complete monotonicity of
\(\sum a_nx^n/\sum b_nx^n\). A second-coefficient obstruction produces entire
counterexamples with strictly completely monotone---indeed geometric---ratio
sequences (Counterexample~\ref{cx:transfer}). Thus strict complete
monotonicity of the coefficient-ratio sequence alone does not force complete
monotonicity of the quotient.

The recurring technical engines are few: canonical products over Bessel
zeros; the Riccati reduction through the modified Bessel equation; compact
certificates whose arithmetic is isolated in Appendix~\ref{app:certificates};
half-order closed forms as parameter-endpoint probes; and a
second-coefficient obstruction for quotients. Each is isolated as a lemma so
that the applications read uniformly.

\section{Preliminaries}\label{sec:prelim}

We recall the standard vocabulary; see \cite{schilling-song-vondracek} for
Bernstein-function theory, \cite{ince} for the Riccati equation in the
standard logarithmic-derivative sense, and \cite{watson,dlmf} for Bessel
facts.

\begin{definition}[Completely monotone and Bernstein classes]\label{def:classes}
Let all functions below be defined on \((0,\infty)\) and take values in
\([0,\infty)\).
\begin{enumerate}[label=\textup{(\roman*)},leftmargin=2.3em]
\item A function \(f\) is \emph{completely monotone}, written \(f\in\CM\), if
\(f\in C^\infty\) and \((-1)^n f^{(n)}(x)\ge0\) for every \(n\ge0\) and
\(x>0\).
\item A function \(g\) is a \emph{Bernstein function}, written \(g\in\BF\), if
\(g\in C^\infty\) and \(g'\in\CM\).
\end{enumerate}
A sequence \((c_n)_{n\ge0}\) is \emph{completely monotone} if
\(\Delta^kc_n\ge0\) for all \(n,k\ge0\), where \(\Delta c_n=c_n-c_{n+1}\), and
\emph{strictly} so if all these inequalities are strict.
\end{definition}

\noindent
For sequences this is the Hausdorff finite-difference convention for moment
sequences on \([0,1]\) \cite{hausdorff,widder}.

\noindent
By the Bernstein--Widder theorem, \(f\in\CM\) exactly when \(f\) is the
Laplace transform of a positive measure on \([0,\infty)\)
\cite{widder,schilling-song-vondracek}; in this paper only the elementary
sign conditions of Definition~\ref{def:classes} are used.

\noindent
Two standard facts are used repeatedly; both are in
\cite{schilling-song-vondracek}.

\begin{proposition}[Closure properties]\label{prop:closure}
\begin{enumerate}[label=\textup{(\roman*)},leftmargin=2.3em]
\item If \(f,g\in\BF\), then \(f\circ g\in\BF\).
\item \(x\mapsto x^\alpha\) belongs to \(\BF\) if and only if
\(0\le\alpha\le1\).
\end{enumerate}
\end{proposition}

\noindent
The refutation logic below is elementary and we record it once.

\begin{lemma}[One-point obstruction]\label{lem:one-point}
\begin{enumerate}[label=\textup{(\roman*)},leftmargin=2.3em]
\item A completely monotone function is nonincreasing and convex; a single
certified point where \(F'>0\) or \(F''<0\) proves \(F\notin\CM\).
\item Let \(A\in\R\), \(B,C\ge0\), and \(h(r)=A-Br-Cr^2\). Then \(h\) is
nonincreasing on \([0,\infty)\); hence if \(0<r<\rho\) and \(h(\rho)>0\), then
\(h(r)>h(\rho)>0\). If \(A,B,C,\rho\) are rational, the conclusion rests on
one rational inequality.
\end{enumerate}
\end{lemma}
\begin{proof}
(i) is immediate from the sign conditions defining \(\CM\). (ii) follows from
\(h'(r)=-B-2Cr\le0\) on \([0,\infty)\).
\end{proof}

\subsection{Bessel background}

The functions used throughout are the standard modified Bessel functions.
The function \(I_\nu\) is the modified Bessel function of the first kind of
order \(\nu\), and \(K_\nu\) is the modified Bessel function of the second
kind of order \(\nu\). We use the standard normalization in which
\(K_\nu(u)>0\) for \(u>0\), \(K_{-\nu}=K_\nu\), and
\(K_\nu(u)\sim(\pi/(2u))^{1/2}e^{-u}\) as \(u\to\infty\)
\cite{watson,dlmf}. For \(\nu>-1\) and \(z>0\), \(I_\nu\) has the series
normalization
\begin{equation}\label{eq:I-series}
I_\nu(z)=\sum_{k\ge0}\frac{(z/2)^{2k+\nu}}{k!\,\Gamma(k+\nu+1)},
\end{equation}
and both functions solve the modified Bessel equation
\begin{equation}\label{eq:ode}
u^2w''(u)+uw'(u)-(u^2+\nu^2)w(u)=0.
\end{equation}
We use the recurrence \cite[\S3.71]{watson}
\begin{equation}\label{eq:recurrence}
I_\mu'(z)=I_{\mu-1}(z)-\frac{\mu}{z}I_\mu(z),
\qquad
I_\nu'(z)=I_{\nu+1}(z)+\frac{\nu}{z}I_\nu(z),
\end{equation}
the half-order closed form
\begin{equation}\label{eq:halforder}
K_{1/2}(u)=\sqrt{\frac{\pi}{2u}}\,e^{-u},
\end{equation}
the canonical product over the positive zeros \(j_{\mu,n}\) of the ordinary
Bessel function of the first kind \(J_\mu\) \cite[\S15.41]{watson},
\begin{equation}\label{eq:product}
I_\mu(z)=\frac{(z/2)^\mu}{\Gamma(\mu+1)}
\prod_{n=1}^{\infty}\left(1+\frac{z^2}{j_{\mu,n}^2}\right),
\qquad \mu>-1,
\end{equation}
with \(j_{\mu,n}\sim\pi n\) as \(n\to\infty\), and the large-argument
expansion \cite[\S10.40]{dlmf}
\begin{equation}\label{eq:K-asymptotic}
\frac{K_\nu'(u)}{K_\nu(u)}=-1-\frac{1}{2u}+O(u^{-2}),
\qquad u\to\infty.
\end{equation}
For \(0<\nu<1\) the small-argument behavior is
\(K_\nu(u)\sim2^{\nu-1}\Gamma(\nu)u^{-\nu}\), and for \(\nu=0\),
\(K_0(u)\sim\log(2/u)-\gamma\) with \(K_0'=-K_1\sim-1/u\) \cite[\S10.30]{dlmf}.

Throughout we write, for \(u>0\),
\begin{equation}
r_\nu(u)=\frac{I_\nu'(u)}{I_\nu(u)},
\qquad
q_\nu(u)=\frac{I_{\nu+1}(u)}{I_\nu(u)},
\qquad
y_\nu(u)=\frac{uK_\nu'(u)}{K_\nu(u)},
\end{equation}
so that \(r_\nu=q_\nu+\nu/u\) by \eqref{eq:recurrence}.

\subsection{Riccati normal forms}

The next two lemmas are the reduction announced in the introduction: the
modified Bessel equation \eqref{eq:ode} replaces second derivatives by a
quadratic in the logarithmic derivative.

\begin{lemma}[Riccati \(I_\nu\) normal form]\label{lem:riccati-I}
Let \(\nu\ge0\), and define, for \(u>0\),
\[
g_\nu(u)=\log(\sqrt u\,I_\nu(u)).
\]
Then
\begin{equation}\label{eq:riccati-r}
g_\nu''(u)
=1+\frac{\nu^2-\frac12}{u^2}-\frac{r_\nu(u)}{u}-r_\nu(u)^2
\end{equation}
and, equivalently, in terms of the ratio \(q_\nu\),
\begin{equation}\label{eq:riccati-q}
g_\nu''(u)
=1-q_\nu(u)^2-\frac{(2\nu+1)\,q_\nu(u)}{u}-\frac{\nu+\frac12}{u^2}.
\end{equation}
In particular, the following three statements are equivalent at any fixed
\((\nu,u)\): \(g_\nu''(u)<0\); the Riccati inequality
\(1+(\nu^2-\tfrac12)u^{-2}-r_\nu u^{-1}-r_\nu^2<0\); and the quadratic lower
bound \(q_\nu^2+(2\nu+1)q_\nu u^{-1}+(\nu+\tfrac12)u^{-2}>1\).
\end{lemma}
\begin{proof}
Since \(g_\nu(u)=\tfrac12\log u+\log I_\nu(u)\),
\begin{equation}
g_\nu''(u)=r_\nu'(u)-\frac{1}{2u^2}.
\end{equation}
The equation \eqref{eq:ode} for \(I_\nu\) gives
\begin{equation}
\frac{I_\nu''(u)}{I_\nu(u)}=1+\frac{\nu^2}{u^2}-\frac{r_\nu(u)}{u},
\end{equation}
and \(r_\nu'=I_\nu''/I_\nu-r_\nu^2\), whence \eqref{eq:riccati-r}. For
\eqref{eq:riccati-q}, substitute \(r_\nu=q_\nu+\nu/u\) into
\eqref{eq:riccati-r}:
\begin{equation}
\begin{aligned}
g_\nu''
&=1+\frac{\nu^2-\frac12}{u^2}
-\frac{q_\nu}{u}-\frac{\nu}{u^2}
-q_\nu^2-\frac{2\nu q_\nu}{u}-\frac{\nu^2}{u^2}\\
&=1-q_\nu^2-\frac{(2\nu+1)q_\nu}{u}-\frac{\nu+\frac12}{u^2}.
\end{aligned}
\end{equation}
The three equivalences are then transcriptions of the sign of the same
quantity.
\end{proof}

\begin{lemma}[Riccati normal form for the \(K\)-quotient]\label{lem:riccati-K}
Let \(\nu\ge0\) and \(g_\nu=K_\nu'/K_\nu^2\). Then, for \(u>0\),
\begin{equation}\label{eq:K-quotient-derivative}
g_\nu'(u)=\frac{F_\nu(u)}{u^2K_\nu(u)},
\qquad
F_\nu(u)=u^2+\nu^2-y_\nu(u)-2y_\nu(u)^2.
\end{equation}
Since \(K_\nu>0\), the sign of \(g_\nu'\) is the sign of \(F_\nu\).
\end{lemma}
\begin{proof}
Differentiating \(g_\nu=K_\nu'/K_\nu^2\),
\begin{equation}
g_\nu'=\frac{K_\nu K_\nu''-2(K_\nu')^2}{K_\nu^3}.
\end{equation}
The equation \eqref{eq:ode} for \(K_\nu\) gives
\(K_\nu''=(1+\nu^2/u^2)K_\nu-K_\nu'/u\), so
\begin{equation}
\frac{K_\nu K_\nu''-2(K_\nu')^2}{K_\nu^2}
=1+\frac{\nu^2}{u^2}-\frac{y_\nu}{u^2}-\frac{2y_\nu^2}{u^2}
=\frac{F_\nu(u)}{u^2}.
\end{equation}
Dividing by \(K_\nu\) gives \eqref{eq:K-quotient-derivative}.
\end{proof}

\section{Bernstein positivity from the zero expansion}
\label{sec:bernstein-positivity}

Yang and Tian \cite{yang-tian} study the ratio
\begin{equation}
W_\nu(z)=\frac{zI_\nu(z)}{I_{\nu+1}(z)},\qquad \nu>-1,\ z>0,
\end{equation}
and propose a conjecture for the Bernstein-function behavior of
\(W_\nu(x^{1/\theta})\) in the range \(\theta\ge2\)
\cite[Conjecture~3]{yang-tian}. With \(\tau=1/\theta\), the
Bernstein-property formulation treated here asks whether
\(x\mapsto W_\nu(x^\tau)\) belongs to \(\BF\) for \(0<\tau\le1/2\). By
\eqref{eq:recurrence}, \(W_\nu(z)=zr_{\nu+1}(z)+(\nu+1)\) is an affine shift
of the logarithmic derivative of \(I_{\nu+1}\), so the question is a
Bernstein-positivity statement for the same object that drives the
refutations of Section~\ref{sec:certificate-cluster}.

\begin{proposition}[Zero partial fraction]\label{prop:partial-fraction}
For every \(\nu>-1\), with \(j_{\nu+1,n}\) the positive zeros of
\(J_{\nu+1}\), the function \(G_\nu(s)=W_\nu(\sqrt s)\) satisfies
\begin{equation}\label{eq:G-partial-fraction}
G_\nu(s)=2(\nu+1)+2\sum_{n=1}^{\infty}\frac{s}{s+j_{\nu+1,n}^2},
\qquad
G_\nu'(s)=2\sum_{n=1}^{\infty}\frac{j_{\nu+1,n}^2}{(s+j_{\nu+1,n}^2)^2},
\end{equation}
for \(s>0\). Consequently \(G_\nu\in\BF\) on \((0,\infty)\).
\end{proposition}
\begin{proof}
Put \(\mu=\nu+1>0\). Logarithmic differentiation of the canonical product
\eqref{eq:product}, justified locally uniformly because
\(j_{\mu,n}\sim\pi n\) makes \(\sum_n j_{\mu,n}^{-2}\) finite, gives
\begin{equation}
z\frac{I_\mu'(z)}{I_\mu(z)}
=\mu+2\sum_{n=1}^{\infty}\frac{z^2}{z^2+j_{\mu,n}^2}.
\end{equation}
The recurrence \eqref{eq:recurrence} then yields
\begin{equation}
W_\nu(z)=z\frac{I_{\mu-1}(z)}{I_\mu(z)}
=z\frac{I_\mu'(z)}{I_\mu(z)}+\mu
=2\mu+2\sum_{n=1}^{\infty}\frac{z^2}{z^2+j_{\mu,n}^2}.
\end{equation}
Substituting \(s=z^2\) gives the first display in
\eqref{eq:G-partial-fraction}. Termwise differentiation is justified by the
same locally uniform convergence, giving the second display. For every
\(k\ge0\) and every \(n\),
\begin{equation}
(-1)^k\frac{d^k}{ds^k}\,
\frac{j_{\nu+1,n}^2}{(s+j_{\nu+1,n}^2)^2}
=(k+1)!\,\frac{j_{\nu+1,n}^2}{(s+j_{\nu+1,n}^2)^{k+2}}\ge0,
\end{equation}
and the locally uniformly convergent sum preserves these inequalities. Hence
\(G_\nu'\in\CM\) and \(G_\nu\in\BF\).
\end{proof}

\begin{theorem}[Power-Bernstein theorem for \(W_\nu\)]\label{thm:yang-tian}
For every \(\nu>-1\) and every \(\tau\in(0,1/2]\), the function
\(x\mapsto W_\nu(x^\tau)\) is a Bernstein function on \((0,\infty)\).
\end{theorem}
\begin{proof}
Write \(W_\nu(x^\tau)=G_\nu(x^{2\tau})\) with \(G_\nu\) as in
Proposition~\ref{prop:partial-fraction}. Since \(0<2\tau\le1\), the power
\(x\mapsto x^{2\tau}\) is a Bernstein function by
Proposition~\ref{prop:closure}(ii), and \(G_\nu\in\BF\) by
Proposition~\ref{prop:partial-fraction}. Composition closure,
Proposition~\ref{prop:closure}(i), concludes.
\end{proof}

\begin{proposition}[Sharpness of the exponent range]\label{prop:sharpness}
For every \(\nu>-1\) and every \(\tau>1/2\), the function
\(x\mapsto W_\nu(x^\tau)\) is not a Bernstein function on \((0,\infty)\).
\end{proposition}
\begin{proof}
Let \(\sigma=2\tau>1\) and \(F(x)=G_\nu(x^{\sigma})\), so
\(F'(x)=\sigma x^{\sigma-1}G_\nu'(x^{\sigma})\). By \eqref{eq:G-partial-fraction}, \(G_\nu'\) has the finite right limit
\begin{equation}
G_\nu'(0^+)=2\sum_{n=1}^{\infty}j_{\nu+1,n}^{-2}\in(0,\infty),
\end{equation}
where finiteness and passage to the limit follow from
\(j_{\nu+1,n}\sim\pi n\) and dominated convergence. Hence
\(F'(x)\sim\sigma G_\nu'(0^+)\,x^{\sigma-1}\to0\) as \(x\downarrow0\),
while \(F'(x)>0\) for every \(x>0\). If \(F'\) were completely monotone, then
it would be nonincreasing by Lemma~\ref{lem:one-point}(i). The finite
right-limit \(F'(0^+)=0\) would then force \(F'(x)\le0\) for all \(x>0\),
contrary to the displayed positivity. Hence \(F'\notin\CM\), and
\(F\notin\BF\).
\end{proof}

\begin{remark}[Source problem]
Theorem~\ref{thm:yang-tian} proves the Bernstein-property formulation
corresponding to the Yang--Tian range \cite[Conjecture~3]{yang-tian}, and
Proposition~\ref{prop:sharpness} shows that range is exactly the Bernstein
regime: \(W_\nu(x^\tau)\in\BF\) if and only if \(0<\tau\le1/2\). The
mechanism is the first half of the organizing thesis: the logarithmic
derivative of \(I_{\nu+1}\) has a partial-fraction expansion over the squared
zeros \(j_{\nu+1,n}^2\), and this expansion exhibits \(G_\nu'\) as a positive
superposition of completely monotone kernels. The representation carries the
whole sign pattern directly.
\end{remark}

\section{One certificate for log-concavity}
\label{sec:certificate-cluster}

Question~1 of Baricz, Ponnusamy, and Vuorinen (BPV) asks whether
\(u\mapsto\sqrt u\,I_\nu(u)\) is strictly log-concave on \((0,\infty)\) for
every \(\nu\ge0\) \cite[Question~1]{bpv}. Lemma~\ref{lem:riccati-I}
identifies this question with the sign of the Riccati expression
\eqref{eq:riccati-r}, and with the corresponding quadratic lower bound for
the ratio \(q_\nu=I_{\nu+1}/I_\nu\). A later coefficient-based treatment
discusses log-concavity statements for \(t^\mu I_\nu(t)\), including the
square-root case \((\nu,\mu)=(0,1/2)\) \cite[Theorem~4]{nanthanasub}. The
certificate below gives an exact obstruction at \(\nu=0\). Thus the single
certificate stated in Lemma~\ref{lem:certificate}, with the arithmetic
isolated in Appendix~\ref{app:certificates}, refutes the BPV source question
and the two algebraic reformulations displayed below.

\begin{lemma}[Rational certificate at \(u=10\)]\label{lem:certificate}
\begin{equation}
r_0(10)=q_0(10)=\frac{I_1(10)}{I_0(10)}<\frac{9487}{10000}.
\end{equation}
\end{lemma}
\begin{proof}
This is the first assertion of Proposition~\ref{prop:certificate-arithmetic}.
The appendix contains the finite rational enclosure and its exact verification.
No floating-point approximation enters the certificate.
\end{proof}

\begin{theorem}[The Riccati expression is positive at \((0,10)\)]
\label{thm:riccati-positive}
With \(g_0(u)=\log(\sqrt u\,I_0(u))\),
\begin{equation}
g_0''(10)
=\frac{199}{200}-\frac{r_0(10)}{10}-r_0(10)^2
\ge\frac{9831}{100000000}>0.
\end{equation}
\end{theorem}
\begin{proof}
The first equality is \eqref{eq:riccati-r} at \(\nu=0\), \(u=10\). The map
\(h(r)=\tfrac{199}{200}-\tfrac{r}{10}-r^2\) is nonincreasing for \(r\ge0\)
by Lemma~\ref{lem:one-point}(ii), and \(0<r_0(10)<9487/10000\) by
Lemma~\ref{lem:certificate}, so
\begin{equation}
g_0''(10)=h(r_0(10))>h\!\left(\frac{9487}{10000}\right)
=\frac{9831}{10^8},
\end{equation}
where the last rational evaluation is
Proposition~\ref{prop:certificate-arithmetic}(ii). This proves the claim.
\end{proof}

\begin{counterexample}[Square-root log-concavity fails]
\label{cx:sqrt-log-concavity}
The statement ``for every \(\nu\ge0\), the function
\(u\mapsto\sqrt u\,I_\nu(u)\) is strictly log-concave on \((0,\infty)\)''
is false. It fails at \(\nu=0\), \(u=10\), and, by continuity, on an open set
of parameters \((\nu,u)\) with \(\nu>0\).
\end{counterexample}
\begin{proof}
Strict log-concavity at \(u\) is \(g_\nu''(u)<0\), and
Theorem~\ref{thm:riccati-positive} certifies \(g_0''(10)>0\). For the open
set: locally uniform convergence of the defining series for \(I_\nu\) and
\(\partial_u I_\nu\) on a neighborhood of \(\nu=0\) contained in
\((-1,\infty)\), together with positivity of \(I_\nu(10)\), gives continuity
of \(\nu\mapsto r_\nu(10)\), so
\(\nu\mapsto g_\nu''(10)=1+(\nu^2-\tfrac12)10^{-2}-r_\nu(10)/10-r_\nu(10)^2\)
is continuous and positive at \(\nu=0\) with certified lower bound
\(9831\times10^{-8}\); hence \(g_\nu''(10)>0\) for all \(\nu\) in some
interval \([0,\nu_0)\), and \(g_\nu''>0\) persists on a neighborhood of
\(u=10\) for each such \(\nu\).
\end{proof}

\begin{counterexample}[Riccati log-concavity inequality fails]
\label{cx:riccati-inequality}
The inequality
\begin{equation}
1+\frac{\nu^2-\frac12}{u^2}-\frac{r_\nu(u)}{u}-r_\nu(u)^2<0
\end{equation}
does not hold for every \(\nu\ge0\) and \(u>0\): it fails at
\((\nu,u)=(0,10)\).
\end{counterexample}
\begin{proof}
By Lemma~\ref{lem:riccati-I}, the displayed expression equals
\(g_\nu''(u)\); apply Theorem~\ref{thm:riccati-positive}.
\end{proof}

\begin{counterexample}[Quadratic ratio lower bound fails]
\label{cx:quadratic-bound}
The claimed bound
\begin{equation}
q_\nu(u)^2+\frac{(2\nu+1)\,q_\nu(u)}{u}+\frac{\nu+\frac12}{u^2}>1
\end{equation}
does not hold for every \(\nu\ge0\) and \(u>0\). At \((\nu,u)=(0,10)\),
\begin{equation}
q_0(10)^2+\frac{q_0(10)}{10}+\frac{1}{200}
\le1-\frac{9831}{10^8}<1.
\end{equation}
\end{counterexample}
\begin{proof}
By \eqref{eq:riccati-q}, the left side equals \(1-g_\nu''(u)\) at every
\((\nu,u)\); at \((0,10)\), Theorem~\ref{thm:riccati-positive} gives
\(1-g_0''(10)\le1-9831\times10^{-8}<1\).
\end{proof}

\begin{remark}[Source problems and orientation values]
Counterexamples~\ref{cx:sqrt-log-concavity}--\ref{cx:quadratic-bound} answer
BPV Question~1 and its algebraic reformulations \cite{bpv}. They also rule out
any all-domain square-root log-concavity interpretation of
\cite[Theorem~4]{nanthanasub} at the parameter pair \((\nu,\mu)=(0,1/2)\).
The exact certificate is isolated in
Appendix~\ref{app:certificates}. The orientation values
\eqref{eq:orientation-decimals} show that the obstruction is small, which
explains why floating-point exploration can misclassify this region.
\end{remark}

\section{The Bessel-\texorpdfstring{\(K\)}{K} quotient: classification and
the endpoint question}
\label{sec:besselk}

BPV prove that \(u\mapsto K_\nu'(u)/K_\nu(u)^2\) is strictly decreasing on
\((0,\infty)\) when \(|\nu|\ge1\) \cite[Theorem~2(a)]{bpv}, and BPV
Question~2 asks what happens on the remaining range \(|\nu|<1\). For recent
adjacent monotonicity and complete-monotonicity results for functions built
from modified Bessel \(K\)-ratios and products, see
\cite{yang-zheng-k,yang-chu-k,segura-2021,mao-tian-k}. The Riccati normal form of
Lemma~\ref{lem:riccati-K} reduces the question to the endpoint signs of
\(F_\nu\), which are decided by the classical expansions.

\begin{proposition}[Endpoint signs in the open range]\label{prop:K-nonmonotone}
For every \(|\nu|<1\), the derivative of
\(g_\nu(u)=K_\nu'(u)/K_\nu(u)^2\) is positive for all sufficiently small
\(u>0\) and negative for all sufficiently large \(u\). Consequently
\(g_\nu\) has no global monotonicity on \((0,\infty)\).
\end{proposition}
\begin{proof}
Since \(K_{-\nu}=K_\nu\), it suffices to treat \(0\le\nu<1\). By
Lemma~\ref{lem:riccati-K}, the sign of \(g_\nu'\) is the sign of
\(F_\nu(u)=u^2+\nu^2-y_\nu(u)-2y_\nu(u)^2\).

\emph{Small \(u\), case \(0<\nu<1\).} The expansion
\(K_\nu(u)\sim2^{\nu-1}\Gamma(\nu)u^{-\nu}\) gives
\(y_\nu(u)\to-\nu\) as \(u\downarrow0\), hence
\begin{equation}
F_\nu(u)\to\nu^2+\nu-2\nu^2=\nu(1-\nu)>0,
\end{equation}
so \(g_\nu'>0\) near \(0\).

\emph{Small \(u\), case \(\nu=0\).} From
\(K_0(u)\sim\log(2/u)-\gamma\) and \(K_0'=-K_1\sim-1/u\),
\begin{equation}
y_0(u)\sim-\frac{1}{\log(2/u)-\gamma}\longrightarrow0^-.
\end{equation}
Writing \(a(u)=-y_0(u)>0\), we have \(a(u)<1/2\) for all small \(u\), so
\begin{equation}
F_0(u)=u^2+a(u)\bigl(1-2a(u)\bigr)>0,
\end{equation}
and again \(g_0'>0\) near \(0\).

\emph{Large \(u\).} The expansion \eqref{eq:K-asymptotic} gives
\(y_\nu(u)=-u-\tfrac12+O(u^{-1})\), so
\begin{equation}
F_\nu(u)=u^2+\nu^2-y_\nu(u)-2y_\nu(u)^2=-u^2-u+\nu^2+O(1)<0
\end{equation}
for all sufficiently large \(u\). The derivative has opposite signs near the
two endpoints, so \(g_\nu\) is neither nondecreasing nor nonincreasing.
Taking \(a>0\) in the small-\(u\) region and \(b>a\) in the large-\(u\)
region, the maximum of \(g_\nu\) on \([a,b]\) is attained at an interior
point. Hence every open-range quotient has at least one interior local
maximum.
\end{proof}

\begin{proposition}[Half-order quotient is unimodal]
\label{prop:halforder-unimodal}
At \(\nu=1/2\), the quotient \(g(u)=K_{1/2}'(u)/K_{1/2}(u)^2\) satisfies
\(g'(u)>0\) on \(\bigl(0,\tfrac{\sqrt2-1}{2}\bigr)\) and \(g'(u)<0\) on
\(\bigl(\tfrac{\sqrt2-1}{2},\infty\bigr)\); its unique critical point
\(u=(\sqrt2-1)/2\) is a strict global maximum.
\end{proposition}
\begin{proof}
Put \(C=\sqrt{\pi/2}\). By \eqref{eq:halforder},
\(K_{1/2}'(u)/K_{1/2}(u)=-1-1/(2u)\), so
\begin{equation}
g(u)=\frac{K_{1/2}'(u)/K_{1/2}(u)}{K_{1/2}(u)}
=-C^{-1}e^{u}\left(u^{1/2}+\tfrac12u^{-1/2}\right).
\end{equation}
Differentiating,
\begin{equation}
g'(u)
=-C^{-1}e^{u}\left(u^{1/2}+u^{-1/2}-\tfrac14u^{-3/2}\right)
=-C^{-1}e^{u}u^{-3/2}\left(u^2+u-\tfrac14\right).
\end{equation}
The quadratic \(u^2+u-\tfrac14\) has the single positive root
\(r=(\sqrt2-1)/2\), is negative on \((0,r)\), and positive on
\((r,\infty)\); since the prefactor \(-C^{-1}e^{u}u^{-3/2}\) is strictly
negative, the stated signs of \(g'\) follow.
\end{proof}

\begin{theorem}[Full monotonicity classification]\label{thm:K-classification}
Let \(g_\nu(u)=K_\nu'(u)/K_\nu(u)^2\). Then \(g_\nu\) is strictly decreasing on
\((0,\infty)\) if and only if \(|\nu|\ge1\). For \(|\nu|\ge1\) this is BPV
Theorem~2(a) \cite{bpv}; for \(|\nu|<1\), \(g_\nu\) increases near \(0\)
and decreases near \(\infty\), hence is nonmonotone.
\end{theorem}
\begin{proof}
Combine the source theorem \cite[Theorem~2(a)]{bpv} for \(|\nu|\ge1\) with
Proposition~\ref{prop:K-nonmonotone} for \(|\nu|<1\); the symmetry
\(K_{-\nu}=K_\nu\) reduces everything to \(\nu\ge0\).
\end{proof}

\begin{remark}[Source problem]
Theorem~\ref{thm:K-classification} closes the monotonicity question left
open in \cite{bpv} for the quotient \(K_\nu'/K_\nu^2\): the source's
decreasing range \(|\nu|\ge1\) is exactly the monotone range.
Proposition~\ref{prop:halforder-unimodal} sharpens the picture on the
half-order line, where the unique interior maximum is explicit. The mechanism is
again the Riccati normal form: monotonicity is the universal negativity of
the quadratic \(F_\nu\) in \(y_\nu\), and the small-\(u\) limit
\(y_\nu\to-\nu\) turns the left endpoint into the parabola
\(\nu(1-\nu)\), positive precisely on the open range \(0<\nu<1\).
\end{remark}

We now turn to BPV Question~7, which asks whether
\(u\mapsto u^2K_\nu'(u)\) is strictly decreasing on \((0,2)\) for every
\(|\nu|\le1/2\) \cite[Question~7]{bpv}, after the source proves the corresponding statement on
\((0,1)\).

\begin{counterexample}[Endpoint failure for Baricz--Ponnusamy--Vuorinen Question~7]
\label{cx:question7}
At \(\nu=1/2\), the derivative of \(u\mapsto u^2K_\nu'(u)\) is positive on
\(\bigl(\tfrac{1+\sqrt2}{2},\,2\bigr)\). Thus the claimed strict decrease on
\((0,2)\) fails at the endpoint order, and BPV Question~7
\cite[Question~7]{bpv} has a negative answer.
\end{counterexample}
\begin{proof}
Put \(c=\sqrt{\pi/2}\), so that \(K_{1/2}(u)=c\,u^{-1/2}e^{-u}\) by
\eqref{eq:halforder}. Differentiating,
\begin{equation}
\begin{aligned}
K_{1/2}'(u)
&=-\,c\,e^{-u}\left(u^{-1/2}+\tfrac12u^{-3/2}\right),\\
u^2K_{1/2}'(u)
&=-c\,e^{-u}\left(u^{3/2}+\tfrac12u^{1/2}\right).
\end{aligned}
\end{equation}
Differentiating once more,
\begin{equation}
\frac{d}{du}\left\{u^2K_{1/2}'(u)\right\}
=c\,e^{-u}u^{-1/2}\left(u^2-u-\tfrac14\right).
\end{equation}
The quadratic \(u^2-u-\tfrac14\) has positive root \((1+\sqrt2)/2<2\) and is
positive beyond it, so \(u^2K_{1/2}'(u)\) is strictly increasing on
\(((1+\sqrt2)/2,\,2)\) and cannot be strictly decreasing on \((0,2)\). The
order \(\nu=1/2\) lies in the admissible set \(|\nu|\le1/2\), so the
universal claim fails.
\end{proof}

\begin{remark}[Source problem]
The refutation lives at the parameter endpoint \(\nu=1/2\), where \(K_\nu\)
degenerates to the elementary closed form \eqref{eq:halforder} and the sign
analysis becomes a quadratic in \(u\). The two positive roots
\((\sqrt2\pm1)/2\) appearing in Proposition~\ref{prop:halforder-unimodal}
and Counterexample~\ref{cx:question7} are the two branches of the same
half-order Riccati algebra. The source's interval \((0,1)\) is safe because
\((1+\sqrt2)/2>1\); the extension to \((0,2)\) crosses the root and fails.
\end{remark}

\section{The limits of transfer: coefficient ratios do not control quotients}
\label{sec:transfer}

The positive results of Section~\ref{sec:bernstein-positivity} depend on the
zero structure of \(I_\nu\). Taylor-coefficient positivity alone is
insufficient, as shown by Baricz's transfer problem
\cite[Problem~4.3(c)]{baricz-bounds}: if
\(f(x)=\sum_{n\ge0}a_nx^n\) and \(g(x)=\sum_{n\ge0}b_nx^n\) converge on
\((-r,r)\), with \(b_n>0\), and the ratio sequence
\(\{a_n/b_n\}_{n\ge0}\) is strictly completely monotone, respectively
completely monotone, must \(f/g\) be strictly completely monotone,
respectively completely monotone, on \((0,r)\)? The motivation is visible in
\eqref{eq:I-series}: many normalized Bessel quotients can be written as
quotients of entire power series whose coefficient ratios are Hausdorff
moment sequences \cite{hausdorff}. An affirmative transfer principle would
generate ratio complete-monotonicity theorems by inspection. The question
strictly strengthens the classical Biernacki--Krzy\.z monotonicity criterion
for quotients of power series with monotone coefficient ratios
\cite{biernacki-krzyz}.

\begin{proposition}[Second-coefficient obstruction]\label{prop:second-coeff}
Let \(g(x)=1+px+qx^2+O(x^3)\) be analytic at \(0\), let \(0<t<1\), and set
\(f(x)=g(tx)\). Then
\begin{equation}
\frac{f(x)}{g(x)}
=1+p(t-1)x+(1-t)\left\{p^2-q(1+t)\right\}x^2+O(x^3).
\end{equation}
In particular, if \(p^2<q(1+t)\), then \((f/g)''(0)<0\), and \(f/g\) is not
completely monotone on any interval \((0,\varepsilon)\).
\end{proposition}
\begin{proof}
With \(f(x)=1+ptx+qt^2x^2+O(x^3)\) and
\(1/g(x)=1-px+(p^2-q)x^2+O(x^3)\),
\begin{equation}
\frac fg=1+p(t-1)x+\left\{qt^2-p^2t+p^2-q\right\}x^2+O(x^3),
\end{equation}
and \(qt^2-q+p^2(1-t)=(1-t)\{p^2-q(1+t)\}\). If this coefficient is
negative, then \((f/g)''(0)=2(1-t)\{p^2-q(1+t)\}<0\), and by continuity
\((f/g)''<0\) on some \((0,\varepsilon)\). A completely monotone function is
convex (Lemma~\ref{lem:one-point}(i)), so \(f/g\notin\CM\) there.
\end{proof}

\begin{counterexample}[Coefficient-ratio \(\CM\) transfer is false]
\label{cx:transfer}
Define
\begin{equation}
g(x)=1+\frac x4+x^2+\sum_{n\ge3}\frac{x^n}{n!},
\qquad
f(x)=g(x/2).
\end{equation}
Then \(f,g\) are entire, all coefficients \(b_n\) of \(g\) are positive, and
the ratio sequence \(a_n/b_n=2^{-n}\) is strictly completely monotone; yet
\(f/g\) is not completely monotone on any interval \((0,\varepsilon)\). Hence
the transfer problem of \cite{baricz-bounds} has a negative answer.
\end{counterexample}
\begin{proof}
The coefficients of \(g\) are \(b_0=1\), \(b_1=1/4\), \(b_2=1\), and
\(b_n=1/n!\) for \(n\ge3\), all positive; \(g\) differs from \(e^x\) by a
polynomial, so \(f\) and \(g\) are entire. Since \(f(x)=g(x/2)\), the
coefficients satisfy \(a_n=2^{-n}b_n\), so \(c_n=a_n/b_n=2^{-n}\) and
\begin{equation}
\Delta^kc_n=\left(1-\tfrac12\right)^k2^{-n}=2^{-n-k}>0
\qquad\text{for all }n,k\ge0,
\end{equation}
a strictly completely monotone sequence (indeed the Hausdorff moment
sequence \cite{hausdorff} of the point mass at \(1/2\)). Applying
Proposition~\ref{prop:second-coeff} with \(p=1/4\), \(q=1\), \(t=1/2\),
\begin{equation}
[x^2]\frac fg=\frac12\left(\frac1{16}-\frac32\right)=-\frac{23}{32},
\qquad
\left(\frac fg\right)''(0)=-\frac{23}{16}<0,
\end{equation}
so \(f/g\) is not completely monotone on any right neighborhood of \(0\).
\end{proof}

\begin{remark}[Source problem]
Counterexample~\ref{cx:transfer} gives a negative answer to Baricz's
coefficient-ratio problem \cite[Problem~4.3(c)]{baricz-bounds}. The ratio sequence here is
as strong as sequence-level hypotheses get---geometric, hence a Hausdorff
moment sequence with strict complete monotonicity of every order---and the
quotient still fails at the second Taylor coefficient. The obstruction
isolates why the Bernstein positivity of
Section~\ref{sec:bernstein-positivity} runs through canonical products over
zeros. Coefficientwise positivity alone is insufficient for the required
complete-monotonicity pattern.
\end{remark}

\section{Concluding remarks}

The tally is one Yang--Tian Bernstein-property formulation proved with a sharp
exponent range
(Theorem~\ref{thm:yang-tian} and Proposition~\ref{prop:sharpness}), one BPV
monotonicity classification completed (Theorem~\ref{thm:K-classification},
with the explicit half-order maximum of
Proposition~\ref{prop:halforder-unimodal}), and three source problems answered
negatively: BPV Question~1, BPV Question~7, and Baricz Problem~4.3(c)
(Counterexamples~\ref{cx:sqrt-log-concavity}--\ref{cx:quadratic-bound},
\ref{cx:question7}, and~\ref{cx:transfer}). All items are statements about
logarithmic derivatives of modified Bessel functions, and all proofs pass
through the same Riccati normal forms
(Lemmas~\ref{lem:riccati-I} and~\ref{lem:riccati-K}).

The division of labor is the one announced in the introduction. On the true
side, the partial-fraction expansion over squared Bessel zeros
(Proposition~\ref{prop:partial-fraction}) carries the entire sign pattern,
and composition closure converts it into the stated Bernstein property.
On the false side, the quadratic normal form concentrates each claim into
the value of a ratio at one point, and Appendix~\ref{app:certificates}
supplies the finite rational certificate: truncated positive series with
geometric tails, compared as integers. The certified positive margin at
\((\nu,u)=(0,10)\) is small. This explains why these problems sit close to
the boundary of truth: the refuted inequalities hold on large parts of the
parameter space and fail by less than \(3\times10^{-4}\) at the certified
obstruction.

Three threads seem worth pursuing. First, the exact log-concavity region of
\(\sqrt u\,I_\nu(u)\) in the \((\nu,u)\)-plane is a natural next object: the
certificate above excludes a neighborhood of \((0,10)\), and the boundary is
the zero set of the Riccati quadratic in Lemma~\ref{lem:riccati-I}. Second,
the endpoint-sign method of
Section~\ref{sec:besselk} should classify the remaining monotonicity windows
for \(K_\nu\)-quotients in the BPV list, with the half-order line as the
canonical probe. Third, Counterexample~\ref{cx:transfer} invites a
characterization question: which additional structure on the coefficient
ratio or denominator coefficients restores the transfer? Any positive answer
would interpolate between the coefficient-level and zero-level mechanisms
that this paper keeps deliberately separate.

\appendix
\renewcommand{\thetheorem}{\Alph{section}.\arabic{theorem}}

\section{Exact rational certificates}\label{app:certificates}

This appendix contains the finite rational arithmetic invoked in
Lemma~\ref{lem:certificate} and Theorem~\ref{thm:riccati-positive}.

\begin{lemma}[Geometric tail enclosure]\label{lem:tail}
Let \(t_k>0\) and suppose \(t_{k+1}/t_k\le q<1\) for all \(k\ge N+1\). Then
\begin{equation}\label{eq:tail-enclosure}
\sum_{k=0}^{N}t_k<\sum_{k=0}^{\infty}t_k
\le\sum_{k=0}^{N}t_k+\frac{t_{N+1}}{1-q}.
\end{equation}
If the \(t_k\) and \(q\) are rational, both enclosing bounds are explicit
rationals.
\end{lemma}
\begin{proof}
The lower bound is positivity of the omitted terms. For the upper bound,
induction gives \(t_{N+1+j}\le t_{N+1}q^{j}\) for \(j\ge0\), and the geometric
series sums to \(t_{N+1}/(1-q)\).
\end{proof}

\begin{proposition}[Exact arithmetic package]\label{prop:certificate-arithmetic}
The following two rational statements hold:
\begin{enumerate}[label=\textup{(\roman*)},leftmargin=2.3em]
\item \(I_1(10)/I_0(10)<9487/10000\).
\item If \(h(r)=199/200-r/10-r^2\), then
\(h(9487/10000)=9831/10^8\).
\end{enumerate}
\end{proposition}
\begin{proof}
From \eqref{eq:I-series},
\begin{equation}\label{eq:appendix-bessel-series}
I_0(10)=\sum_{k=0}^{\infty}\frac{25^k}{(k!)^2},
\qquad
I_1(10)=\sum_{k=0}^{\infty}\frac{5\cdot25^k}{k!\,(k+1)!}.
\end{equation}
Both series have positive rational terms. Set
\begin{equation}\label{eq:appendix-L0}
L_0=\sum_{k=0}^{20}\frac{25^k}{(k!)^2}<I_0(10).
\end{equation}
For the \(I_1\) terms, for \(k\ge21\),
\begin{equation}\label{eq:appendix-tail-ratio}
\frac{t_{k+1}}{t_k}
=\frac{25}{(k+1)(k+2)}
\le \frac{25}{22\cdot23}
=\frac{25}{506}.
\end{equation}
Therefore Lemma~\ref{lem:tail} gives
\begin{equation}\label{eq:appendix-U1}
I_1(10)\le U_1:=
\sum_{k=0}^{20}\frac{5\cdot25^k}{k!\,(k+1)!}
+\frac{5\cdot25^{21}}{21!\,22!}\cdot\frac{506}{481}.
\end{equation}
The exact rational comparison is
\begin{equation}\label{eq:appendix-integer-witness}
9487L_0-10000U_1
=
\frac{9066027287067976494607015156102735792727}
{3214199764297248897429588860517482496}>0.
\end{equation}
Hence
\begin{equation}\label{eq:appendix-ratio-certificate}
\frac{I_1(10)}{I_0(10)}
<\frac{U_1}{L_0}
<\frac{9487}{10000},
\end{equation}
which proves (i). For (ii), direct rational evaluation gives
\begin{equation}\label{eq:appendix-h-value}
h\!\left(\frac{9487}{10000}\right)
=\frac{199}{200}-\frac{9487}{100000}
-\left(\frac{9487}{10000}\right)^2
=\frac{9831}{10^8}.
\end{equation}
\end{proof}

\begin{remark}[Decimal orientation]\label{rem:appendix-decimals}
For orientation only,
\begin{equation}\label{eq:orientation-decimals}
\begin{aligned}
I_0(10)&=2815.7166\ldots,\qquad
I_1(10)=2670.9883\ldots,\\
r_0(10)&=0.948599\ldots,\qquad
g_0''(10)=2.99\ldots\times10^{-4}.
\end{aligned}
\end{equation}
These decimals are not used in the proofs.
\end{remark}

\section*{Disclosure}

The author declares no competing interests and received no specific funding
for this work. No empirical data were used. Every numerical claim reduces to
finite exact rational arithmetic, reproducible from the formulas in
Appendix~\ref{app:certificates}.
During manuscript preparation, the author used ChatGPT and Codex to support
literature triage, organization, language revision, and consistency checks;
the author reviewed the content, verified the mathematical claims and cited
sources, and takes full responsibility for the manuscript.


\begin{thebibliography}{99}

\bibitem{amos}
D.~E.~Amos,
\emph{Computation of modified Bessel functions and their ratios},
Math.\ Comp.\ \textbf{28} (1974), 239--251.

\bibitem{baricz-bounds}
\'A.~Baricz,
\emph{Bounds for modified Bessel functions of the first and second kinds},
Proc.\ Edinb.\ Math.\ Soc.\ \textbf{53} (2010), 575--599.

\bibitem{baricz-turan-modified}
\'A.~Baricz,
\emph{Tur\'an type inequalities for modified Bessel functions},
Bull.\ Aust.\ Math.\ Soc.\ \textbf{82} (2010), 254--264.

\bibitem{biernacki-krzyz}
M.~Biernacki and J.~Krzy\.z,
\emph{On the monotonicity of certain functionals in the theory of analytic
functions},
Ann.\ Univ.\ Mariae Curie-Sk\l odowska Sect.\ A \textbf{9} (1955), 135--147.

\bibitem{bpv}
\'A.~Baricz, S.~Ponnusamy, and M.~Vuorinen,
\emph{Functional inequalities for modified Bessel functions},
Expo.\ Math.\ \textbf{29} (2011), 399--414, arXiv:1009.4814.
\url{https://arxiv.org/abs/1009.4814}.

\bibitem{grosswald}
E.~Grosswald,
\emph{The Student \(t\)-distribution of any degree of freedom is infinitely
divisible},
Z.\ Wahrscheinlichkeitstheorie verw.\ Gebiete \textbf{36} (1976), 103--109.

\bibitem{hausdorff}
F.~Hausdorff,
\emph{Summationsmethoden und Momentfolgen. I},
Math.\ Z.\ \textbf{9} (1921), 74--109.

\bibitem{hartman-watson}
P.~Hartman and G.~S.~Watson,
\emph{``Normal'' distribution functions on spheres and the modified Bessel
functions},
Ann.\ Probab.\ \textbf{2} (1974), 593--607.

\bibitem{hornik-gruen-2013}
K.~Hornik and B.~Gr\"un,
\emph{Amos-type bounds for modified Bessel function ratios},
J.\ Math.\ Anal.\ Appl.\ \textbf{408} (2013), 91--101.

\bibitem{hornik-gruen-2024}
K.~Hornik and B.~Gr\"un,
\emph{Generalized Amos-type bounds for modified Bessel function ratios},
Math.\ Inequal.\ Appl.\ \textbf{27} (2024), 775--787.
\url{https://doi.org/10.7153/mia-2024-27-55}.

\bibitem{ifantis-siafarikas}
E.~K.~Ifantis and P.~D.~Siafarikas,
\emph{Inequalities involving Bessel and modified Bessel functions},
J.\ Math.\ Anal.\ Appl.\ \textbf{147} (1990), 214--227.

\bibitem{ince}
E.~L.~Ince,
\emph{Ordinary Differential Equations},
Dover, New York, 1956.

\bibitem{ismail}
M.~E.~H.~Ismail,
\emph{Integral representations and complete monotonicity of various
quotients of Bessel functions},
Canad.\ J.\ Math.\ \textbf{29} (1977), 1198--1207.

\bibitem{ismail-1990}
M.~E.~H.~Ismail,
\emph{Complete monotonicity of modified Bessel functions},
Proc.\ Amer.\ Math.\ Soc.\ \textbf{108} (1990), 353--361.
\url{https://doi.org/10.1090/S0002-9939-1990-0993753-9}.

\bibitem{ismail-kelker}
M.~E.~H.~Ismail and D.~H.~Kelker,
\emph{Special functions, Stieltjes transforms and infinite divisibility},
SIAM J.\ Math.\ Anal.\ \textbf{10} (1979), 884--901.

\bibitem{joshi-bissu}
C.~M.~Joshi and S.~K.~Bissu,
\emph{Some inequalities of Bessel and modified Bessel functions},
J.\ Aust.\ Math.\ Soc.\ Ser.\ A \textbf{50} (1991), 333--342.

\bibitem{laforgia-natalini}
A.~Laforgia and P.~Natalini,
\emph{Some inequalities for modified Bessel functions},
J.\ Inequal.\ Appl.\ \textbf{2010} (2010), Article ID 253035, 10 pp.

\bibitem{mao-tian-k}
Z.-X.~Mao and J.-F.~Tian,
\emph{Monotonicity and complete monotonicity of some functions involving the
modified Bessel functions of the second kind},
Comptes Rendus Math\'ematique \textbf{361} (2023), 217--235.
\url{https://doi.org/10.5802/crmath.399}.

\bibitem{nasell}
I.~N\r{a}sell,
\emph{Rational bounds for ratios of modified Bessel functions},
SIAM J.\ Math.\ Anal.\ \textbf{9} (1978), 1--11.

\bibitem{nanthanasub}
T.~Nanthanasub, B.~Novaprateep, and N.~Wichailukkana,
\emph{The logarithmic concavity of modified Bessel functions of the first kind
and its related functions},
Adv. Difference Equ. \textbf{2019} (2019), Paper No.~379.
\url{https://doi.org/10.1186/s13662-019-2309-8}.

\bibitem{dlmf}
F.~W.~J.~Olver, D.~W.~Lozier, R.~F.~Boisvert, and C.~W.~Clark (eds.),
\emph{NIST Handbook of Mathematical Functions},
Cambridge University Press, 2010.

\bibitem{schilling-song-vondracek}
R.~L.~Schilling, R.~Song, and Z.~Vondra\v{c}ek,
\emph{Bernstein Functions: Theory and Applications},
2nd ed., de Gruyter, 2012.

\bibitem{robert}
C.~Robert,
\emph{Modified Bessel functions and their applications in probability and
statistics},
Statist.\ Probab.\ Lett.\ \textbf{9} (1990), 155--161.

\bibitem{ruiz-antolin-segura}
D.~Ruiz-Antol\'in and J.~Segura,
\emph{A new type of sharp bounds for ratios of modified Bessel functions},
J.\ Math.\ Anal.\ Appl.\ \textbf{443} (2016), 1232--1246.

\bibitem{segura}
J.~Segura,
\emph{Bounds for ratios of modified Bessel functions and associated
Tur\'an-type inequalities},
J.\ Math.\ Anal.\ Appl.\ \textbf{374} (2011), 516--528.

\bibitem{segura-2021}
J.~Segura,
\emph{Monotonicity properties for ratios and products of modified Bessel
functions and sharp trigonometric bounds},
Results Math.\ \textbf{76} (2021), Article No.~221.
\url{https://doi.org/10.1007/s00025-021-01531-1}.

\bibitem{simpson-spector}
H.~C.~Simpson and S.~J.~Spector,
\emph{Some monotonicity results for ratios of modified Bessel functions},
Quart.\ Appl.\ Math.\ \textbf{42} (1984), 95--98.

\bibitem{turan}
P.~Tur\'an,
\emph{On the zeros of the polynomials of Legendre},
\v{C}asopis P\v{e}st. Mat. Fys. \textbf{75} (1950), 113--122.

\bibitem{watson}
G.~N.~Watson,
\emph{A Treatise on the Theory of Bessel Functions},
2nd ed., Cambridge University Press, 1944.

\bibitem{widder}
D.~V.~Widder,
\emph{The Laplace Transform},
Princeton University Press, 1946.

\bibitem{yang-chu-k}
Z.-H.~Yang and Y.-M.~Chu,
\emph{Monotonicity and inequalities involving the modified Bessel functions of
the second kind},
J.\ Math.\ Anal.\ Appl.\ \textbf{508} (2022), Article No.~125889, 23 pp.

\bibitem{yang-tian}
Z.-H.~Yang and J.-F.~Tian,
\emph{Convexity of ratios of the modified Bessel functions of the first kind
with applications},
Rev.\ Mat.\ Complut.\ \textbf{36} (2023), 799--825.
\url{https://doi.org/10.1007/s13163-022-00439-w}.

\bibitem{yang-zheng-i}
Z.-H.~Yang and S.-Z.~Zheng,
\emph{Monotonicity and convexity of the ratios of the first kind modified
Bessel functions and applications},
Math.\ Inequal.\ Appl.\ \textbf{21} (2018), 107--125.

\bibitem{yang-zheng-k}
Z.-H.~Yang and S.-Z.~Zheng,
\emph{The monotonicity and convexity for the ratios of modified Bessel
functions of the second kind and applications},
Proc.\ Amer.\ Math.\ Soc.\ \textbf{145} (2017), 2943--2958.

\end{thebibliography}
\end{document}